\documentclass[11pt,french,english]{amsart}
\usepackage[T1]{fontenc}
\usepackage[latin9]{inputenc}
\usepackage[a4paper]{geometry}
\usepackage{prettyref}
\usepackage{amsmath}
\usepackage{amssymb}

\makeatletter
\usepackage{latexsym}
\theoremstyle{plain}
\newtheorem{stthm}{Theorem}[section]

\numberwithin{equation}{section} 
\numberwithin{figure}{section} 
\numberwithin{table}{section} 

\def\newrefformat#1#2{%
  \@namedef{pr@#1}##1{#2}}
\newrefformat{eq}{\textup{(\ref{#1})}}
\newrefformat{ssec}{sub-section \ref{#1}}
\newrefformat{sssec}{sub-sub-section \ref{#1}}
\newrefformat{cha}{chapter \ref{#1}}
\newrefformat{sec}{section \ref{#1}}
\newrefformat{tab}{table \ref{#1} on page \pageref{#1}}
\newrefformat{fig}{figure \ref{#1} on page \pageref{#1}}
\def\prettyref#1{\@prettyref#1|}
\def\@prettyref#1|#2|{%
  \expandafter\ifx\csname pr@#1\endcsname\relax%
    \PackageWarning{prettyref}{Reference format #1\space undefined}%
    \ref{#1|#2}%
  \else%
    \csname pr@#1\endcsname{#1|#2}%
  \fi%
}

\DeclareRobustCommand{\QED}{%
  \ifmmode 
  \else \leavevmode\unskip\penalty9999 \hbox{}\nobreak\hfill
  \fi
  \quad\hbox{\QEDsymbol}}
\newcommand{\QEDsymbol}{Q.E.D.}

 \theoremstyle{definition}
 \newtheorem{stdefn}[stthm]{Definition}
 \theoremstyle{plain}
 \newtheorem{stprop}[stthm]{Proposition}
 \theoremstyle{plain}
 \newtheorem{stlem}[stthm]{Lemma}
 \theoremstyle{remark}
 \newtheorem*{claim*}{Claim}
 \theoremstyle{plain}
 \newtheorem{stcor}[stthm]{Corollary}
 \theoremstyle{remark}
 \newtheorem{strem}[stthm]{Remark}

\usepackage[bookmarksopen, naturalnames]{hyperref}
\newrefformat{thm}{Theorem \ref{#1}}
\newrefformat{prop}{Proposition \ref{#1}}
\newrefformat{cor}{Corollary \ref{#1}}
\newrefformat{lem}{Lemma \ref{#1}}
\newrefformat{rem}{Remark \ref{#1}}
\newrefformat{sec}{Section \ref{#1}}
\newrefformat{sub}{Subsection \ref{#1}}
\newrefformat{subsub}{Paragraph \ref{#1}}
\newrefformat{def}{Definition \ref{#1}}

\newenvironment{myitemize} {\begin{list}{}{%
\setlength{\labelwidth}{1em}
\setlength{\labelsep}{0pt}
\setlength{\leftmargin}{1em}
}}%
{\end{list}}

\makeatother

\usepackage{babel}
\addto\extrasfrench{\providecommand{\fg}{\ifdim\lastskip>\z@\unskip\fi~\frqq}}

\begin{document}

\title{Composition operators on weighted Bergman-Orlicz spaces on the ball}

\author{Stéphane Charpentier}
\begin{abstract}
We give embedding theorems for weighted Bergman-Orlicz spaces on the
ball and then apply our results to the study of the boundedness and
the compactness of composition operators in this context. As one of
the motivations of this work, we show that there exist some weighted
Bergman-Orlicz spaces, different from $H^{\infty}$, on which every
composition operator is bounded.
\end{abstract}

\subjclass[2000]{Primary: 47B33 - Secondary: 32C22; 46E15}

\keywords{Bergman-Orlicz space - Carleson measure - Composition operator}

\address{Charpentier Stéphane, Département de Mathématiques, Bâtiment 425,
Université Paris-Sud, F-91405, Orsay, France}

\email{stephane.charpentier@math.u-psud.fr}

\maketitle
\def\leftmark{\MakeUppercase{Composition operators on Bergman-Orlicz
and Hardy-Orlicz spaces}}

\def\rightmark{\MakeUppercase{Stéphane Charpentier}}

\section{Introduction and preliminaries}

\subsection{Introduction}

Let $\mathbb{B}_{N}$ denote the unit ball in $\mathbb{C}^{N}$ and
$\phi$ an analytic map from $\mathbb{B}_{N}$ into itself. In this
paper, we are interesting in characterizing the continuity and the
compactness of composition operators $C_{\phi}$, defined by $C_{\phi}\left(f\right)=f\circ\phi$,
on weighted Bergman-Orlicz spaces. On the classical weighted Bergman
spaces $A_{\alpha}^{p}\left(\mathbb{B}_{N}\right)$, or on the Hardy
spaces $H^{p}\left(\mathbb{B}_{N}\right)$ as well, the boundedness
or the compactness of $C_{\phi}$ can be characterized in terms of
Carleson measures (see e.g. \cite{COWEN-MACCLUER}). In one variable,
the Littlewood subordination principle is known to be the main tool
to show that composition operators are always bounded on these spaces,
whereas B. MacCluer and J. Shapiro exhibited self-maps $\phi$ on
$\mathbb{B}_{N}$ ($N>1$) inducing non-bounded composition operators
on $A_{\alpha}^{p}\left(\mathbb{B}_{N}\right)$ or on $H^{p}\left(\mathbb{B}_{N}\right)$.
As for the compactness, the same authors gave an example of a surjective
analytic self-map of $\mathbb{D}$ defining a compact composition
operators on these spaces (\cite{MACCLUER-SHAPIRO}). In comparison,
it is easy to check that every $C_{\phi}$ is bounded on $H^{\infty}$
and is compact if and only if $\left\Vert \phi\right\Vert _{\infty}<1$,
whatever $N\geq1$. This arises the question: what is the behavior
of composition operators on significant spaces between $H^{\infty}$
and $A_{\alpha}^{p}\left(\mathbb{B}_{N}\right)$ (or $H^{p}\left(\mathbb{B}_{N}\right)$)?

This question motivated P. Lefèvre, D. Li, H. Queffélec and L. Rodr\'iguez-Piazza
to start, since 2006, a systematic study of composition operators
on Bergman-Orlicz spaces $A^{\psi}\left(\mathbb{D}\right)$ and Hardy-Orlicz
spaces $H^{p}\left(\mathbb{D}\right)$ on the unit disk of $\mathbb{C}$
(e.g. \cite{QUEF-LI-LE-RO-PI,QUEF-LI-LEF-RO-B-O-H-O,QUEF-LI-LEF-ROD-H-2-H-O-COMP-OP,QUEF-LI-LEFEVRE-BERGMAN-ORLICZ}).
Indeed, these spaces reveals to be a satisfying intermediate scale
of spaces between $H^{\infty}$ and the classical Bergman or Hardy
spaces, depending on the growth of the Orlicz function $\psi$. As
a part of their work, they gave an analytic surjective self-map $\phi:\mathbb{D}\rightarrow\mathbb{D}$
such that $C_{\phi}$ is compact on $H^{p}\left(\mathbb{D}\right)$,
extending the preceding result by MacCluer and Shapiro. They partially
solved the same problem in the context of Bergman-Orlicz spaces, by
underlying the fact that the compactness of $C_{\phi}$ on some Hardy-Orlicz
spaces implies the compactness of $C_{\phi}$ on the correspondant
Bergman-Orlicz spaces. By the way, they prove that it is unlikely
to find Orlicz functions $\psi$ such that compactness of composition
operators on $A^{\psi}\left(\mathbb{D}\right)$ (and definitely on
$H^{p}\left(\mathbb{D}\right)$) should be equivalent to that on $H^{\infty}$.

Yet, looking at the several variables setting, the same kind of question
arises, but now even for continuity, since there exists symbol $\phi$
such that $C_{\phi}$ is not bounded on the classical Bergman spaces
$A^{p}\left(\mathbb{B}_{N}\right)$, although every $C_{\phi}$ is
bounded on $H^{\infty}$. The purpose of this paper is to investigate
this problem for weighted Bergman-Orlicz spaces, that is to answer
the question: does there exist some Orlicz function $\psi$ such that
every composition operator is bounded on the weighted Bergman-Orlicz
space $A_{\alpha}^{\psi}\left(\mathbb{B}_{N}\right)$? To do this,
we need to characterize boundedness of composition operators on Bergman-Orlicz
spaces, in a general enough fashion. By passing, we give a characterization
of the compactness of $C_{\phi}$ on $A_{\alpha}^{\psi}\left(\mathbb{B}_{N}\right)$,
which may arise new questions and provide eventually a better understanding
of the behavior of composition operators on these spaces.

We have to mention that, in 2010, Z. J. Jiang gave embedding theorems
and characterizations of the boundedness and the compactness of composition
operators on Bergman-Orlicz spaces $A^{\psi}\left(\mathbb{B}_{N}\right)$
when $\psi$ satisfies the so-called $\Delta_{2}$-Condition (\cite{JIANG}).
This condition somehow implies that the space $A^{\psi}\left(\mathbb{B}_{N}\right)$
is {}``closed'' to a classical\textbf{ }Bergman space and, as we
could guess, these characterizations are the same than that known
for Bergman spaces; their applications to composition operators do
not provide different results from that obtained in the classical
framework; especially, they give no information for {}``small''
Bergman-Orlicz spaces, in which we are especially interesting in.

This paper is organized as follows: after introducing the notions
and materials in Section 1, we give, in section 2, general embedding
theorems for weighted Bergman-Orlicz spaces. Precisely, given two
arbitrary Orlicz functions $\psi_{1}$ and $\psi_{2}$, we exhibit
in \prettyref{thm|Embed_Thm_Berg_Cont} and \prettyref{thm|Embed_Berg_orlizc_Comp}
necessary and sufficient conditions on a measure $\mu$ on the ball
under which the canonical embedding $A_{\alpha}^{\psi_{1}}\left(\mathbb{B}_{N}\right)\hookrightarrow L^{\psi_{2}}\left(\mu\right)$
holds or is compact. In general, we do not get characterizations,
yet we see that we do when $\psi_{1}=\psi_{2}$ satisfies some convenient
regular conditions. In Section 3, applications are given to composition
operators and, as a consequence, we exhibit a class of Orlicz functions
defining weighted Bergman-Orlicz spaces on which every composition
operator is bounded.

\subsection{\label{sub|Orlicz_spaces_def}Orlicz spaces - Notations}

\subsubsection{Definitions}

In this whole paper, we denote by $\psi:\mathbb{R}_{+}\rightarrow\mathbb{R}_{+}$
an Orlicz function, i.e. a strictly convex function vanishing at $0$,
continuous at $0$ and satisfying\[
{\displaystyle \frac{\psi(x)}{x}}\xrightarrow[x\rightarrow\infty]{}+\infty.\]
Note that an Orlicz function is non-decreasing. Considering a probability
space $\left(\Omega,\mathbb{P}\right)$, we define the Orlicz space
$L^{\psi}\left(\Omega\right)$ as the space of all (equivalence classes
of) measurable complex functions $f$ on $\Omega$ for which there
is a constant $C>0$ such that\[
\int_{\Omega}\psi\left(\frac{\left|f\right|}{C}\right)d\mathbb{P}<\infty.\]
This space may be normalized by the Luxemburg norm\[
\left\Vert f\right\Vert _{\psi}=\inf\left\{ C>0,\,\int_{\Omega}\psi\left(\frac{\left|f\right|}{C}\right)d\mathbb{P}\leq1\right\} ,\]
which makes $\left(L^{\psi}\left(\Omega\right),\left\Vert .\right\Vert _{\psi}\right)$
a Banach space such that $L^{\infty}\left(\Omega\right)\subset L^{\psi}\left(\Omega\right)\subset L^{1}\left(\Omega\right)$.
Observe that if $\psi(x)=x^{p}$ for every $x$, then $L^{\psi}\left(\Omega\right)=L^{p}\left(\Omega\right)$.
It is usual to introduce the Morse-Transue space $M^{\psi}\left(\Omega\right)$,
which is the subspace of $L^{\psi}\left(\Omega\right)$ generated
by $L^{\infty}\left(\Omega\right)$.

To every Orlicz function $\psi$, we shall associate its complementary
function $\Phi:\mathbb{R}_{+}\rightarrow\mathbb{R}_{+}$ defined by\[
\Phi(y)=\sup_{x\in\mathbb{R}_{+}}\left\{ xy-\psi(x)\right\} .\]
We may verify that $\Phi$ is also an Orlicz function (see \cite{RAO-REN},
Section 1.3). If both $L^{\Phi}\left(\Omega\right)$ and $L^{\psi}\left(\Omega\right)$
are normed by the Luxemburg norm, then $L^{\Phi}\left(\Omega\right)$
is isomorphic to the dual of $M^{\psi}\left(\Omega\right)$ (\cite[IV, 4.1, Theorem 7]{RAO-REN}).

\subsubsection{Three classes of Orlicz functions}

We now introduce essentially three classes of Orlicz functions which
will appear several times in this paper. This part may appear a little
bit technical, but we would like to convince the reader that this
classification, which permits to get a meaningful scale of Orlicz
spaces between $L^{\infty}$ and $L^{p}$, is quite natural.

\smallskip{}
\begin{myitemize}\item The first class is that of Orlicz functions
which satisfy the so-called $\Delta_{2}$-Condition which is a condition
of moderate growth.
\begin{stdefn}
\label{def|def_delta_2_condition}Let $\psi$ be an Orlicz function.
We say that $\psi$ satisfies the $\Delta_{2}$-Condition if there
exist $x_{0}>0$ and a constant $K>1$, such that\[
\psi\left(2x\right)\leq K\psi\left(x\right)\]
for any $x\geq x_{0}$.
\end{stdefn}
 For example, $x\longmapsto ax^{p}\left(1+b\log\left(x\right)\right)$,
$p>1$, $a>0$ and $b\geq0$, satisfies the $\Delta_{2}$-Condition.
Corollary 5, Chapter II of \cite{RAO-REN} gives:
\begin{stprop}
\label{prop|grow_delta_2_condition}Let $\psi$ be an Orlicz function
satisfying the $\Delta_{2}$-Condition, then there are some $p>1$
and $C>0$ such that $\psi\left(x\right)\leq Cx^{p}$, for $x$ large
enough. Therefore, $L^{p}\subset L^{\psi}\subset L^{1}$, for some
$p>1$.
\end{stprop}
\smallskip{}
\item The two following conditions are also regular conditions which
are satisfied by most of the Orlicz functions that we are interesting
in.
\begin{stdefn}
\label{def|def_nabla_0_cond}Let $\psi$ be an Orlicz function. We
say that $\psi$ satisfies the $\nabla_{0}$-Condition if there exist
some $x_{0}>0$ and some constant $C\geq1$, such that for every $x_{0}\leq x\leq y$
we have\[
\frac{\psi\left(2x\right)}{\psi\left(x\right)}\leq\frac{\psi\left(2Cy\right)}{\psi\left(y\right)}.\]

\end{stdefn}
We refer to Proposition 4.6 of \cite{QUEF-LI-LE-RO-PI} to verify
that we have the following:
\begin{stprop}
Let $\psi$ be an Orlicz function. Then $\psi$ satisfies the $\nabla_{0}$-Condition
if and only if there exists $x_{0}>0$ such that for every (or equivalently
\emph{one}) $\beta>1$, there exists a constant $C_{\beta}\geq1$
such that\[
\frac{\psi\left(\beta x\right)}{\psi\left(x\right)}\leq\frac{\psi\left(\beta C_{\beta}y\right)}{\psi\left(y\right)}\]
for every $x_{0}\leq x\leq y$.
\end{stprop}
Furthermore, the following class will be of interest for us: $\psi$
satisfies the \emph{uniform} $\nabla_{0}$-Condition if it satisfies
the $\nabla_{0}$-Condition for a constant $C_{\beta}\geq1$ independent
of $\beta>1$.

\smallskip{}
\item Finally, one defines a class of Orlicz functions which grow
fast:
\begin{stdefn}
\label{def|def_cond_delta_up_2}Let $\psi$ be an Orlicz function.
$\psi$ satisfies the $\Delta^{2}$-Condition if and only if there
exist $x_{0}>0$ and a constant $C>0$, such that\[
\psi\left(x\right)^{2}\leq\psi\left(Cx\right),\]
for every $x\geq x_{0}$.
\end{stdefn}
The convexity and the non-decrease of Orlicz functions give the following
proposition, whose content can be found in \cite[Chapter II, Paragraph 2.5, pages 40 and further]{RAO-REN}
or in \cite[Chapter I, Section 6, Paragraph 5]{KRASNO-RUT}:
\selectlanguage{french}%
\begin{stprop}
\label{prop|prop_equiv_delta_up_2}\foreignlanguage{english}{Let $\psi$
be an Orlicz function. The assertions:}
\selectlanguage{english}%
\begin{enumerate}
\item $\psi$ satisfies the $\Delta^{2}$-Condition;
\item There exist $b>1$, $C>0$ and $x_{0}>0$ such that $\psi\left(x\right)^{b}\leq\psi\left(Cx\right)$,
for every $x\geq x_{0}$;
\item For every $b>1$, there exist $C_{b}>0$ and $x_{0,b}>0$ such that
$\psi\left(x\right)^{b}\leq\psi\left(C_{b}x\right)$, for every $x\geq x_{0,b}$.
\end{enumerate}
are equivalent.
\end{stprop}
\selectlanguage{english}%
The next proposition (\cite[Chapter II, Paragraph 2, Proposition 6]{RAO-REN})
shows that an Orlicz function which satisfies the $\Delta^{2}$-Condition
need to have at least an exponential growth.
\selectlanguage{french}%
\begin{stprop}
\label{prop|estim_croiss_delta_up_2}\foreignlanguage{english}{Let
$\psi$ be an Orlicz function which satisfies the $\Delta^{2}$-Condition.
There exist $a>0$ and $x_{0}>0$ such that\[
\psi\left(x\right)\geq e^{ax},\]
for every $x\geq x_{0}$.}
\end{stprop}
\selectlanguage{english}%
If $\psi$ satisfies $\Delta^{2}$-Condition, we shall say that $L^{\psi}\left(\Omega\right)$
is a {}``small'' Orlicz space, i.e. {}``far'' from any $L^{p}\left(\Omega\right)$
and {}``close'' to $L^{\infty}$.\end{myitemize}\smallskip{}
To finish, we recall Proposition 4.7 (2) of \cite{QUEF-LI-LE-RO-PI}:
\begin{stprop}
\label{prop|nabla_0_implies_nabla_2}Let $\psi$ be an Orlicz function.
If $\psi$ satisfies the $\Delta^{2}$-Condition, then it satisfies
the uniform $\nabla_{0}$-Condition.
\end{stprop}
\smallskip{}
Let us notice that for any $1<p<\infty$, every function $x\longmapsto x^{p}$
is an Orlicz function which satisfies the uniform $\nabla_{0}$-Condition,
and then the $\nabla_{0}$-condition. It also satisfies the $\Delta_{2}$-Condition.
Furthermore, for any $a>0$ and $b\geq1$, $x\longmapsto e^{ax^{b}}-1$
belongs to the $\Delta^{2}$-Class (and then to the uniform $\nabla_{0}$-Class),
yet not to the $\Delta_{2}$-one. In addition, the Orlicz functions
which can be written $x\rightarrow\exp\left(a\left(\ln\left(x+1\right)\right)^{b}\right)-1$
for $a>0$ and $b\geq1$, satisfy the $\nabla_{0}$-Condition, but
do not belong to the $\Delta^{2}$-Class.

\smallskip{}
For a complete study of Orlicz spaces, we refer to \cite{KRASNO-RUT}
and to \cite{RAO-REN}. We can also find precise information in context
of composition operators, such as other classes of Orlicz functions
and their link together with, in \cite{QUEF-LI-LE-RO-PI}.

\subsection{Weighted Bergman-Orlicz spaces on $\mathbb{B}_{N}$}

Let $\alpha>-1$ and let $dv_{\alpha}$ be the normalized weighted
Lebesgue measure on $\mathbb{B}_{N}$\[
dv_{\alpha}\left(z\right)=c_{\alpha}\left(1-\left|z\right|^{2}\right)^{\alpha}dv\left(z\right),\]
where $dv$ is the normalized volume Lebesgue measure on $\mathbb{B}_{N}$.
The constant $c_{\alpha}$ is equal to\[
c_{\alpha}=\frac{\Gamma\left(n+\alpha+1\right)}{n!\Gamma\left(\alpha+1\right)}.\]
With the notations of the previous subsection, if $\left(\Omega,\mathbb{P}\right)=\left(\mathbb{B}_{N},dv_{\alpha}\right)$,
then the weighted Bergman-Orlicz space $A_{\alpha}^{\psi}\left(\mathbb{B}_{N}\right)$
on the ball is $H\left(\mathbb{B}_{N}\right)\cap L_{\alpha}^{\psi}\left(\mathbb{B}_{N}\right)$,
where $H\left(\mathbb{B}_{N}\right)$ is the space of holomorphic
functions on $\mathbb{B}_{N}$, and where the subscript $\alpha$
remains that the probabilistic measure is the weighted normalized
measure $dv_{\alpha}$ on $\mathbb{B}_{N}$. We have $A_{\alpha}^{\psi}\left(\mathbb{B}_{N}\right)\subset A_{\alpha}^{1}\left(\mathbb{B}_{N}\right)$
and it is classical to check that, if $A_{\alpha}^{\psi}\left(\mathbb{B}_{N}\right)$
is endowed with the Luxemburg norm $\left\Vert .\right\Vert _{\psi}$,
then it is a Banach space.

\smallskip{}
For $a\in\mathbb{B}_{N}$, we denote by $\delta_{a}$ the point evaluation
functional at $a$. The following proposition infers that $\delta_{a}$
is bounded on every $A_{\alpha}^{\psi}\left(\mathbb{B}_{N}\right)$.
\begin{stprop}
\label{prop|point_eval_func_Bergman}Let $\alpha>-1$ and let $\psi$
be an Orlicz function. Let also $a\in\mathbb{B}_{N}$. Then the point
evaluation functional $\delta_{a}$ at $a$ is bounded on $A_{\alpha}^{\psi}\left(\mathbb{B}_{N}\right)$;
more precisely, we have\[
\frac{1}{4^{N+1+\alpha}}\psi^{-1}\left(\left(\frac{1+\left|a\right|}{1-\left|a\right|}\right)^{N+1+\alpha}\right)\leq\left\Vert \delta_{a}\right\Vert \leq\psi^{-1}\left(\left(\frac{1+\left|a\right|}{1-\left|a\right|}\right)^{N+1+\alpha}\right).\]
\end{stprop}
\begin{proof}
We denote by $H_{a}$ the Berezin kernel at $a$, defined by\[
H_{a}\left(z\right)=\left(\frac{1-\left|a\right|^{2}}{\left|1-\left\langle z,a\right\rangle \right|^{2}}\right)^{N+1+\alpha},\, z\in\mathbb{B}_{N}.\]
It is not hard to check -and well-known- that $\left\Vert H_{a}\right\Vert _{\infty}=\left(\frac{1+\left|a\right|}{1-\left|a\right|}\right)^{N+1+\alpha}$
and that $\left\Vert H_{a}\right\Vert _{L^{1}}=1$. Let $\varphi_{a}$
be an automorphism of $\mathbb{B}_{N}$ such that $\varphi\left(0\right)=a$.
Fix $f\in A_{\alpha}^{\psi}\left(\mathbb{B}_{N}\right)$ and set $C=\left\Vert f\right\Vert _{A_{\alpha}^{\psi}}$.
By the change of variables formula (e.g. \cite{ZHU}, Proposition
1.13), and using the subharmonicity of ${\displaystyle \psi\left(\frac{\left|f\circ\varphi_{a}\right|}{C}\right)}$,
we get\[
\psi\left(\frac{\left|f\left(a\right)\right|}{C}\right)\leq\int_{\mathbb{B}_{N}}\psi\left(\frac{\left|f\circ\varphi_{a}\right|}{C}\right)dv_{\alpha}=\int_{\mathbb{B}_{N}}\psi\left(\frac{\left|f\left(z\right)\right|}{C}\right)H_{a}\left(z\right)dv_{\alpha}\left(z\right).\]
Since $\psi^{-1}$ is non-decreasing, we obtain\[
\left|f\left(a\right)\right|\leq C\psi^{-1}\left(\left(\frac{1+\left|a\right|}{1-\left|a\right|}\right)^{N+1+\alpha}\right),\]
hence the intended upper estimate.

Conversely, we compute $\delta_{a}\left(H_{a}\right)$. It gives\begin{eqnarray}
\left\Vert \delta_{a}\right\Vert  & \geq & \frac{\left|H_{a}\left(a\right)\right|}{\left\Vert H_{a}\right\Vert _{A_{\alpha}^{\psi}}}\nonumber \\
 & \geq & \frac{1}{\left(1-\left|a\right|^{2}\right)^{N+1+\alpha}}\frac{\psi^{-1}\left(\left\Vert H_{a}\right\Vert _{\infty}\right)}{\left\Vert H_{a}\right\Vert _{\infty}}\mbox{ (by \cite[Lemma 3.9]{QUEF-LI-LE-RO-PI})}\nonumber \\
 & \geq & \frac{1}{4^{N+1+\alpha}}\psi^{-1}\left(\left(\frac{1+\left|a\right|}{1-\left|a\right|}\right)^{N+1+\alpha}\right).\label{eq|calc_norm_Berezin_Apsi}\end{eqnarray}

\end{proof}

\section{\label{sub|Berg_Emb_Thm}Embedding Theorems for Bergman-Orlicz spaces}

We will need a version of Carleson's theorem for Bergman spaces slightly
different from the traditional one. This is inspired from \cite{QUEF-LI-LEFEVRE-BERGMAN-ORLICZ}.
Anyway, as for the study of continuity and compactness of composition
operators on Bergman spaces or Hardy spaces of the ball in terms of
Carleson measure, we will need to introduce the objects and notions
involved. We first recall the definition of the non-isotropic distance
on the sphere $\mathbb{S}_{N}$, which we denote by $d$. For $\left(\zeta,\xi\right)\in\mathbb{S}_{N}^{2}$,
it is given by\[
d\left(\zeta,\xi\right)=\sqrt{\left|1-\left\langle \zeta,\xi\right\rangle \right|}.\]
We may verify that the map $d$ is a distance on $\mathbb{S}_{N}$
and can be extended to $\overline{\mathbb{B}_{N}}$, where it still
satisfies the triangle inequality. For $\zeta\in\overline{\mathbb{B}_{N}}$
and $h\in\left]0,1\right]$, we define the non-isotropic {}``ball''
of $\mathbb{B}_{N}$ by\[
S\left(\zeta,h\right)=\left\{ z\in\mathbb{B}_{N},\, d\left(\zeta,z\right)^{2}<h\right\} .\]
and its analogue in $\overline{\mathbb{B}_{N}}$ by\[
\mathcal{S}\left(\zeta,h\right)=\left\{ z\in\overline{\mathbb{B}_{N}},\, d\left(\zeta,z\right)^{2}<h\right\} .\]
Let us also denote by\[
Q=\mathcal{S}\left(\zeta,h\right)\cap\mathbb{S}_{N}\]
the {}``true'' balls in $\mathbb{S}_{N}$. Next, for $\zeta\in\mathbb{S}_{N}$
and $h\in\left]0,1\right]$, we define\[
W\left(\zeta,h\right)=\left\{ z\in\mathbb{B}_{N},\,1-\left|z\right|<h,\,\frac{z}{\left|z\right|}\in Q\left(\zeta,h\right)\right\} .\]
$W\left(\zeta,h\right)$ is called a Carleson window.

We introduce the two following functions $\varrho_{\mu}$ and $K_{\mu,\alpha}$:\[
\varrho_{\mu}\left(h\right)=\sup_{\xi\in\mathbb{S}_{N}}\mu\left(W\left(\xi,h\right)\right)\]
where $\mu$ is positive Borel measure on $\mathbb{B}_{N}$. We now
set\[
K_{\mu,\alpha}\left(h\right)=\sup_{0<t\leq h}\frac{\varrho_{\mu}\left(t\right)}{t^{N+1+\alpha}}.\]
$\mu$ is said to be an $\alpha$-Bergman-Carleson measure if $K_{\mu,\alpha}$
is bounded. As\begin{equation}
t^{N+1+\alpha}\sim v_{\alpha}\left(W\left(\xi,t\right)\right)\label{eq|eq_leb_mes_Carle_win_power_Berg}\end{equation}
for every $\xi\in\mathbb{S}_{N}$, this is equivalent to the existence
of a constant $C>0$ such that\[
\mu\left(W\left(\xi,h\right)\right)\leq Cv_{\alpha}\left(W\left(\xi,h\right)\right)\]
for any $\xi\in\mathbb{S}_{N}$ and any $h\in\left(0,1\right)$ (or
equivalently any $h\in\left(0,h_{A}\right)$ for some $0<h_{A}\leq1$).
Let us remark that, in the definition of $\varrho_{\mu}$ and $K_{\mu,\alpha}$,
we may have taken $S\left(\xi,h\right)$ instead of $W\left(\xi,h\right)$,
since these two sets are equivalent in the sense that there exist
two constants $C_{1}>0$ and $C_{2}>0$ such that\[
S\left(\xi,C_{1}h\right)\subset W\left(\xi,h\right)\subset S\left(\xi,C_{2}h\right).\]
Next we may work indifferently with non-isotropic balls or Carleson
windows if there is no possible confusion.

\smallskip{}
We have the following covering lemma which will be useful for our
version of Carleson's theorem:
\begin{stlem}
\label{lem|covering_lemma_Bergman}There exists an integer $M>0$
such that for any $0<r<1$, we can find a finite sequence $\left\{ \xi_{k}\right\} _{k=1}^{m}$
($m$ depending on $r$) in $\mathbb{S}_{N}$ with the following properties:
\begin{enumerate}
\item $\mathbb{S}_{N}=\bigcup_{k}Q\left(\xi_{k},r\right)$.
\item The sets $Q\left(\xi_{k},r/4\right)$ are mutually disjoint.
\item Each point of $\mathbb{S}_{N}$ belongs to at most $M$ of the sets
$Q\left(\xi_{k},4r\right)$.
\end{enumerate}
\end{stlem}
\begin{proof}
The proof, using a variant of \cite[Lemma 2.22]{ZHU} for the non-isotropic
distance at the boundary is quite identical to that of \cite[Theorem 2.23]{ZHU}.
The fact that we can take a finite union follows from a compactness
argument.
\end{proof}
From now on, $M$ will always stand for the constant involved in \prettyref{lem|covering_lemma_Bergman}.
We will now define a maximal operator associated to a covering of
the ball with convenient subsets. Let $n\geq0$ be an integer and
denote by $C_{n}$ the corona\[
C_{n}=\left\{ z\in\mathbb{B}_{N},\,1-\frac{1}{2^{n}}\leq\left|z\right|<1-\frac{1}{2^{n+1}}\right\} .\]
For any $n\geq0$, let $\left(\xi_{n,k}\right)_{k}\subset\mathbb{S}_{N}$
be given by \prettyref{lem|covering_lemma_Bergman} putting ${\displaystyle r=\frac{1}{2^{n}}}$.
For $k\geq0$, we set\[
T_{0,k}=\left\{ z\in\mathbb{B}_{N}\setminus\left\{ 0\right\} ,\,\frac{z}{\left|z\right|}\in Q\left(\xi_{0,k},1\right)\right\} \cup\left\{ 0\right\} .\]
Then let us define the sets $T_{n,k}$, for $n\geq1$ and $k\geq0$,
by\[
T_{n,k}=\left\{ z\in\mathbb{B}_{N}\setminus\left\{ 0\right\} ,\,\frac{z}{\left|z\right|}\in Q\left(\xi_{n,k},\frac{1}{2^{n}}\right)\right\} .\]
We have both\[
\bigcup_{n\geq0}C_{n}=\mathbb{B}_{N}\]
and\[
\bigcup_{k\geq0}T_{0,k}=\mathbb{B}_{N}\mbox{ and }\bigcup_{k\geq0}T_{n,k}=\mathbb{B}_{N}\setminus\left\{ 0\right\} ,\, n\geq1.\]
For $\left(n,k\right)\in\mathbb{N}^{2}$, we finally define the subset
$\Delta_{\left(n,k\right)}$ of $\mathbb{B}_{N}$ by\[
\Delta_{\left(n,k\right)}=C_{n}\cap T_{n,k}.\]
We have\begin{eqnarray*}
\Delta_{\left(0,k\right)} & = & \left(W\left(\xi_{0,k},1\right)\cap C_{0}\right)\cup\left\{ 0\right\} ;\\
\Delta_{\left(n,k\right)} & = & W\left(\xi_{n,k},\frac{1}{2^{n}}\right)\cap C_{n},\, n\geq1.\end{eqnarray*}
By construction, the $\Delta_{\left(n,k\right)}$'s satisfy the following
properties:
\begin{enumerate}
\item $\bigcup_{\left(n,k\right)\in\mathbb{N}^{2}}\Delta_{\left(n,k\right)}=\mathbb{B}_{N}.$
\item For every $\left(n,k\right)$, $\Delta_{\left(n,k\right)}$ is a subset
of the closed Carleson window $\overline{{\displaystyle W\left(\xi_{n,k},\frac{1}{2^{n}}\right)}}$
and by construction, we can find a constant $\tilde{C}>0$, independent
of $\left(n,k\right)$ such that\[
v_{\alpha}\left({\displaystyle W\left(\xi_{n,k},\frac{1}{2^{n}}\right)}\right)\leq\tilde{C}v_{\alpha}\left(\Delta_{\left(n,k\right)}\right).\]

\item Given $0<\varepsilon<1/2$, if $C_{n}^{\varepsilon}$ denotes the
corona defined by\[
C_{n}^{\varepsilon}=\left\{ z\in\mathbb{B}_{N},\,\left(1+\varepsilon\right)\left(1-\frac{1}{2^{n}}\right)\leq\left|z\right|<\left(1+\varepsilon\right)\left(1-\frac{1}{2^{n+1}}\right)\right\} ,\]
then each point of $\mathbb{B}_{N}$ belongs to at most $M$ of the
sets $\Delta_{\left(n,k\right)}^{\varepsilon}$'s defined by\begin{eqnarray*}
\Delta_{\left(0,k\right)}^{\varepsilon} & = & \left({\displaystyle W\left(\xi_{0,k},1+\varepsilon\right)}\cap C_{0}^{\varepsilon}\right)\cup\left\{ 0\right\} ;\\
\Delta_{\left(n,k\right)}^{\varepsilon} & = & {\displaystyle W\left(\xi_{n,k},\left(1+\varepsilon\right)\frac{1}{2^{n}}\right)}\cap C_{n}^{\varepsilon},\, n\geq1.\end{eqnarray*}
This comes from the construction and the previous covering lemma.
In particular, we have\[
\sum_{\left(n,k\right)\in\mathbb{N}^{2}}v_{\alpha}\left(\Delta_{\left(n,k\right)}^{\varepsilon}\right)\leq Mv_{\alpha}\left(\mathbb{B}_{N}\right)=M.\]

\end{enumerate}
For any $f\in A_{\alpha}^{\psi}\left(\mathbb{B}_{N}\right)$, we define
the following maximal function $\Lambda_{f}$:\begin{equation}
\Lambda_{f}=\sum_{n,k\geq0}\sup_{\Delta_{\left(n,k\right)}}\left(\left|f\left(z\right)\right|\right)\chi_{\Delta_{\left(n,k\right)}}\label{eq|def_max_funct_Bergm_Orlicz}\end{equation}
where $\chi_{\Delta_{\left(n,k\right)}}$ is the characteristic function
of $\Delta_{\left(n,k\right)}$. The next proposition says that the
maximal operator $\Lambda:f\longmapsto\Lambda_{f}$ is bounded from
$A_{\alpha}^{\psi}\left(\mathbb{B}_{N}\right)$ to $L_{\alpha}^{\psi}\left(\mathbb{B}_{N},v_{\alpha}\right)$.
\begin{stprop}
\label{prop|cont_ma_funct_Berg_Orlicz}Let $\psi$ be an Orlicz function
and let $\alpha>-1$. Then the maximal operator $\Lambda$ is bounded
from $A_{\alpha}^{\psi}\left(\mathbb{B}_{N}\right)$ to $L_{\alpha}^{\psi}\left(\mathbb{B}_{N},v_{\alpha}\right)$.
More precisely there exists $B\geq1$ such that for every $f\in A_{\alpha}^{\psi}\left(\mathbb{B}_{N}\right)$,
we have\[
\left\Vert \Lambda_{f}\right\Vert _{L_{\alpha}^{\psi}}\leq2B\left\Vert f\right\Vert _{A_{\alpha}^{\psi}}.\]
\end{stprop}
\begin{proof}
Fix $f\in A_{\alpha}^{\psi}\left(\mathbb{B}_{N}\right)$ and set $C=\left\Vert f\right\Vert _{A_{\alpha}^{\psi}}$.
We denote by ${\displaystyle c_{\left(n,k\right)}=\sup_{\Delta_{\left(n,k\right)}}\left(\left|f\right|\right)}$
and let $\tau_{\left(n,k\right)}\in\Delta_{\left(n,k\right)}$ be
such that ${\displaystyle \left|f\left(\tau_{\left(n,k\right)}\right)\right|\geq\frac{c_{\left(n,k\right)}}{2}}$.
Since ${\displaystyle \frac{\psi\circ\left|f\right|}{C}}$ is subharmonic,
and by a usual refined submean property, we have\begin{eqnarray*}
\int_{\mathbb{B}_{N}}\psi\left(\frac{\Lambda_{f}}{2C}\right)dv_{\alpha} & \leq & \sum_{n,k\geq0}\psi\left(\frac{\left|f\left(\tau_{\left(n,k\right)}\right)\right|}{C}\right)v_{\alpha}\left(\Delta_{\left(n,k\right)}\right)\\
 & \leq & \sum_{n,k\geq0}\frac{v_{\alpha}\left(\Delta_{\left(n,k\right)}\right)}{v_{\alpha}\left(\Delta_{\left(n,k\right)}^{\varepsilon}\right)}\int_{\Delta_{\left(n,k\right)}^{\varepsilon}}\psi\left(\frac{\left|f\right|}{C}\right)dv_{\alpha}.\end{eqnarray*}
A classical computation shows that\[
\frac{v_{\alpha}\left(\Delta_{\left(n,k\right)}\right)}{v_{\alpha}\left(\Delta_{\left(n,k\right)}^{\varepsilon}\right)}\leq D_{\varepsilon},\]
where $D_{\varepsilon}$ is a positive constant which only depends
on $\varepsilon$. Therefore we get,\begin{eqnarray*}
\int_{\mathbb{B}_{N}}\psi\left(\frac{\Lambda_{f}}{2C}\right)dv_{\alpha} & \leq & D_{\varepsilon}\sum_{n,k\geq0}\int_{\Delta_{\left(n,k\right)}^{\varepsilon}}\psi\left(\frac{\left|f\right|}{C}\right)dv_{\alpha}.\end{eqnarray*}
Now, we have $C_{n}^{\varepsilon}=\cup_{k\geq0}\Delta_{\left(n,k\right)}^{\varepsilon}$
and, by construction of the $\Delta_{\left(n,k\right)}$'s, for every
$n$, each point of $C_{n}^{\varepsilon}$ belongs to at most $M$
of the sets $\Delta_{\left(n,k\right)}^{\varepsilon}$. Then, for
$n$ fixed,\[
\sum_{k\geq0}\int_{\Delta_{\left(n,k\right)}^{\varepsilon}}\psi\left(\frac{\left|f\right|}{C}\right)dv_{\alpha}\leq M\int_{C_{n}^{\varepsilon}}\psi\left(\frac{\left|f\right|}{C}\right)dv_{\alpha}.\]
Next, we of course have $\mathbb{B}_{N}\subset\cup_{n\geq0}C_{n}^{\varepsilon}$
and each point of $\mathbb{B}_{N}$ belongs to at most $3$ of the
$C_{n}^{\varepsilon}$'s. It follows that\begin{eqnarray*}
\int_{\mathbb{B}_{N}}\psi\left(\frac{\Lambda_{f}}{2C}\right)dv_{\alpha} & \leq & D_{\varepsilon}M\sum_{n\geq0}\int_{C_{n}^{\varepsilon}}\psi\left(\frac{\left|f\right|}{C}\right)dv_{\alpha}\\
 & \leq & B\int_{\mathbb{B}_{N}}\psi\left(\frac{\left|f\right|}{C}\right)dv_{\alpha}\end{eqnarray*}
for some constant $B\geq1$. Now, by convexity, we get\[
\int_{\mathbb{B}_{N}}\psi\left(\frac{\Lambda_{f}}{2BC}\right)dv_{\alpha}\leq1,\]
hence $\left\Vert \Lambda_{f}\right\Vert _{L_{\alpha}^{\psi}}\leq2B\left\Vert f\right\Vert _{A_{\alpha}^{\psi}}.$
\end{proof}
We state our version of Carleson's theorem as follows:
\begin{stthm}
\label{thm|Bergman_Carleson_Theorem}There exists a constant $\tilde{C}>0$
such that, for every $f\in A_{\alpha}^{1}\left(\mathbb{B}_{N}\right)$
and every positive finite Borel measure $\mu$ on $\mathbb{B}_{N}$,
we have\[
\mu\left(\left\{ z\in\mathbb{B}_{N},\,\left|z\right|>1-h\mbox{ and }\left|f\left(z\right)\right|>t\right\} \right)\leq\tilde{C}K_{\mu,\alpha}\left(2h\right)v_{\alpha}\left(\left\{ \Lambda_{f}>t\right\} \right)\]
for every $h\in\left(0,1/2\right)$ and every $t>0$.\end{stthm}
\begin{proof}
The proof is quite identical to that of \cite[Lemma 2.3]{QUEF-LI-LEFEVRE-BERGMAN-ORLICZ}.
Anyway, we prefer to give the details. Fix $0<h<1$ and $t>0$. We
identify $i\in\mathbb{N}$ and $\left(n,k\right)\in\mathbb{N}^{2}$
thanks to an arbitrary bijection from $\mathbb{N}^{2}$ onto $\mathbb{N}$.
We will write $i\longleftrightarrow\left(n,k\right)$ without possible
confusion. Define\[
I=\left\{ i\longleftrightarrow\left(n,k\right),\,\sup_{\Delta_{i}}\left|f\right|>t\right\} \]
and\[
I_{h}=\left\{ i\longleftrightarrow\left(n,k\right),\, h>\frac{1}{2^{n+1}}\mbox{ and }\sup_{\Delta_{i}}\left|f\right|>t\right\} .\]
Denoting by $W_{i}$ the smallest Carleson window containing $\Delta_{i}$,
by the three properties of the $\Delta_{i}$'s listed above, we can
find some constants $C>0$ and $\tilde{C}>0$ such that\begin{eqnarray*}
\mu\left(\left\{ z\in\mathbb{B}_{N},\,\left|z\right|>1-h\mbox{ and }\left|f\left(z\right)\right|>t\right\} \right) & \leq & \sum_{i\in I_{h}}\mu\left(\Delta_{i}\right)\\
 & \leq & \sum_{i\in I_{h}}\mu\left(W_{i}\right)\\
 & \leq & C\sum_{i\in I_{h}}K_{\mu,\alpha}\left(2h\right)v_{\alpha}\left(W_{i}\right)\\
 & \leq & C\tilde{C}K_{\mu,\alpha}\left(2h\right)\sum_{i\in I}v_{\alpha}\left(\Delta_{i}\right).\end{eqnarray*}
The third inequality comes from \prettyref{eq|eq_leb_mes_Carle_win_power_Berg}
and from the fact that, for every $i\in I_{h}$, as the radius of
$W_{i}$ is smaller than ${\displaystyle \frac{1}{2^{n}}}$, it is
then smaller than $2h$. Now, as each point of $\mathbb{B}_{N}$ belongs
to at most $M$ of the $\Delta_{i}$'s, we have\[
\sum_{i\in I}v_{\alpha}\left(\Delta_{i}\right)\leq Mv_{\alpha}\left(\bigcup_{i\in I}\Delta_{i}\right)\leq Mv_{\alpha}\left(\left\{ \Lambda_{f}>t\right\} \right).\]
and\[
\mu\left(\left\{ z\in\mathbb{B}_{N},\,\left|z\right|>1-h\mbox{ and }\left|f\left(z\right)\right|>t\right\} \right)\lesssim K_{\mu,\alpha}\left(2h\right)v_{\alpha}\left(\left\{ \Lambda_{f}>t\right\} \right).\]

\end{proof}
The last lemma gives the following technical result.
\begin{stlem}
\label{lem|lem_tech_Berg_Orl}Let $\mu$ be a finite positive Borel
measure on $\mathbb{B}_{N}$ and let $\psi_{1}$ and $\psi_{2}$ be
two Orlicz functions. Assume that there exist $A>0$, $\eta>0$ and
$h_{A}\in\left(0,1/2\right)$ such that\[
K_{\mu,\alpha}\left(h\right)\leq\eta\frac{1/h^{N+1+\alpha}}{\psi_{2}\left(A\psi_{1}^{-1}\left(1/h^{N+1+\alpha}\right)\right)}\]
for every $h\in\left(0,h_{A}\right)$. Then, there exist three constants
$B>0$, $x_{A}>0$ and $C_{1}$ (this latter does not depend on $A$,
$\eta$ and $h_{A}$) such that, for every $f\in A_{\alpha}^{\psi_{1}}\left(\mathbb{B}_{N}\right)$
such that $\left\Vert f\right\Vert _{A_{\alpha}^{\psi_{1}}}\leq1$,
and every Borel subset $E$ of $\mathbb{B}_{N}$, we have\[
\int_{E}\psi_{2}\left(\frac{\left|f\right|}{B}\right)d\mu\leq\mu\left(E\right)\psi_{2}\left(x_{A}\right)+C_{1}\eta\int_{\mathbb{B}_{N}}\psi_{1}\left(\Lambda_{f}\right)dv_{\alpha}.\]
\end{stlem}
\begin{proof}
For $f\in A_{\alpha}^{\psi_{1}}\left(\mathbb{B}_{N}\right)$, $\left\Vert f\right\Vert _{A_{\alpha}^{\psi_{1}}}\leq1$,
and $E$ a Borel subset of $\mathbb{B}_{N}$, we begin by writing
the following formula, based on Fubini's integration:\begin{equation}
\int_{E}\psi_{2}\left(\left|f\right|\right)d\mu=\int_{0}^{\infty}\psi_{2}^{'}\left(t\right)\mu\left(\left\{ \left|f\right|>t\right\} \cap E\right)dt.\label{eq|eq_lem_tech_Berg_thm_fubini}\end{equation}
We concentrate our attention on the expression $\mu\left(\left\{ \left|f\right|>t\right\} \right)$.
We use the upper estimate of the point evaluation functional obtained
in \prettyref{prop|point_eval_func_Bergman} to get that if ${\displaystyle \left|f\left(z\right)\right|>t}$,
then, since $\left\Vert f\right\Vert _{A_{\alpha}^{\psi_{1}}}\leq1$,
we have\begin{eqnarray}
t & < & \psi_{1}^{-1}\left(\left(\frac{1+\left|z\right|}{1-\left|z\right|}\right)^{N+1+\alpha}\right)\nonumber \\
 & \leq & 2^{N+1+\alpha}\psi_{1}^{-1}\left(\left(\frac{1}{1-\left|z\right|}\right)^{N+1+\alpha}\right)\label{eq|inequ_eval_lem_tech}\end{eqnarray}
because $\psi$ is a convex function. Inequality \prettyref{eq|inequ_eval_lem_tech}
is now equivalent to the following one:\[
\left|z\right|>1-\left(\frac{1}{\psi_{1}\left(\frac{t}{2^{N+1+\alpha}}\right)}\right)^{1/\left(N+1+\alpha\right)}.\]
Carleson's theorem (\prettyref{thm|Bergman_Carleson_Theorem}) then
yields that\begin{eqnarray}
\mu\left(\left\{ \left|f\right|>t\right\} \right) & = & \mu\left(\left\{ \left|f\right|>t\right\} \cap\left\{ \left|z\right|>1-\left(\frac{1}{\psi_{1}\left(\frac{t}{2^{N+1+\alpha}}\right)}\right)^{1/\left(N+1+\alpha\right)}\right\} \right)\nonumber \\
 & \leq & \tilde{C}K_{\mu,\alpha}\left(2\left(\frac{1}{{\displaystyle \psi_{1}\left(\frac{t}{2^{N+1+\alpha}}\right)}}\right)^{1/\left(N+1+\alpha\right)}\right)v_{\alpha}\left(\left\{ \Lambda_{f}>t\right\} \right).\label{eq|lem_tech_Berg_appli_Carle}\end{eqnarray}
Now, if $A$, $h_{A}$ and $\eta$ are as in the statement of the
lemma, then, if\[
{\displaystyle \frac{1}{2^{N+1+\alpha}}{\displaystyle {\displaystyle \psi_{1}\left(\frac{3.2^{N+\alpha}}{A}s\right)>1/h_{A}^{N+1+\alpha}}}}\]
 i.e. $s\geq x_{A}:={\displaystyle \frac{A}{3.2^{N+\alpha}}\psi_{1}^{-1}\left(\left(2/h_{A}\right)^{N+1+\alpha}\right)}$,
then\begin{equation}
K_{\mu,\alpha}\left(2\left(\frac{1}{{\displaystyle \psi_{1}\left(\frac{3.2^{N+\alpha}}{A}s\right)}}\right)^{1/\left(N+1+\alpha\right)}\right)\leq\frac{\eta}{2^{N+1+\alpha}}\frac{{\displaystyle \psi_{1}\left(\frac{3.2^{N+\alpha}}{A}s\right)}}{\psi_{2}\left(\frac{3}{2}s\right)}.\label{eq|lem_tech_Berg_app_hyp}\end{equation}
Hence, applying \prettyref{eq|eq_lem_tech_Berg_thm_fubini} to ${\displaystyle \frac{A}{6.4^{N+\alpha}}\left|f\right|}$,
together with \prettyref{eq|lem_tech_Berg_appli_Carle} and \prettyref{eq|lem_tech_Berg_app_hyp},
and putting ${\displaystyle t=\frac{6.4^{N+\alpha}}{A}s}$ in \prettyref{eq|lem_tech_Berg_appli_Carle},
we get\begin{multline}
\int_{E}\psi_{2}\left(\frac{A}{6.4^{N+\alpha}}\left|f\right|\right)d\mu\leq\int_{0}^{x_{A}}\psi_{2}^{'}\left(s\right)\mu\left(E\right)ds\\
+\frac{\eta\tilde{C}}{2^{N+1+\alpha}}\int_{x_{A}}^{\infty}\psi_{2}^{'}\left(s\right)\frac{{\displaystyle \psi_{1}\left(\frac{3.2^{N+\alpha}}{A}s\right)}}{\psi_{2}\left(\frac{3}{2}s\right)}v_{\alpha}\left(\left\{ \Lambda_{f}>\frac{6.4^{N+\alpha}}{A}s\right\} \right)ds.\label{eq|last_eq}\end{multline}
For the second integral of the right hand side, notice that for an
Orlicz function $\psi$, we have\[
x\psi^{'}\left(x\right)\leq C\psi\left(\frac{\left(C+1\right)x}{C}\right)\]
for any $C>0$ and any $x\geq0$. Indeed, as $\psi^{'}\left(t\right)$
is non-decreasing, we have\[
\frac{x}{C}\psi^{'}\left(x\right)\leq\int_{x}^{\frac{C+1}{C}x}\psi^{'}\left(t\right)dt\leq\psi\left(\frac{C+1}{C}x\right).\]
Therefore\[
\frac{\psi_{2}^{'}\left(s\right)}{\psi_{2}\left(\frac{3}{2}s\right)}\leq\frac{2}{s}\]
and \prettyref{eq|last_eq} yields\begin{multline*}
\int_{E}\psi_{2}\left(\frac{A}{6.4^{N+\alpha}}\left|f\right|\right)d\mu\leq\psi_{2}\left(x_{A}\right)\mu\left(E\right)\\
+\frac{\eta\tilde{C}}{2^{N+\alpha}}\int_{x_{A}}^{\infty}\frac{1}{s}{\displaystyle \psi_{1}\left(\frac{3.2^{N+\alpha}}{A}s\right)}v_{\alpha}\left(\left\{ \Lambda_{f}>\frac{6.4^{N+\alpha}}{A}s\right\} \right)ds.\end{multline*}
Using the convexity of the function $\psi_{1}$, we get\begin{multline*}
\int_{E}\psi_{2}\left(\frac{A}{6.4^{N+\alpha}}\left|f\right|\right)d\mu\leq\psi_{2}\left(x_{A}\right)\mu\left(E\right)\\
+\frac{\eta\tilde{C}}{2^{N+\alpha}}\frac{3.2^{N+\alpha}}{A}\int_{0}^{\infty}{\displaystyle \psi_{1}^{'}\left(\frac{3.2^{N+\alpha}}{A}s\right)}v_{\alpha}\left(\left\{ \Lambda_{f}>\frac{6.4^{N+\alpha}}{A}s\right\} \right)ds\end{multline*}
i.e.\begin{eqnarray*}
\int_{E}\psi_{2}\left(\frac{A}{6.4^{N+\alpha}}\left|f\right|\right)d\mu & \leq & \psi_{2}\left(x_{A}\right)\mu\left(E\right)+\frac{\eta\tilde{C}}{2^{N+\alpha}}\int_{0}^{\infty}{\displaystyle \psi_{1}^{'}\left(u\right)}v_{\alpha}\left(\left\{ \Lambda_{f}>2^{N+1+\alpha}u\right\} \right)du\\
 & \leq & \psi_{2}\left(x_{A}\right)\mu\left(E\right)+\frac{\eta\tilde{C}}{2.4^{N+\alpha}}\int_{\mathbb{B}_{N}}\psi_{1}\left(\Lambda_{f}\right)dv_{\alpha}\end{eqnarray*}
and the proof of the lemma is complete.
\end{proof}

\subsection{The canonical embedding $A_{\alpha}^{\psi_{1}}\left(\mathbb{B}_{N}\right)\hookrightarrow L^{\psi_{2}}\left(\mu\right)$.}

We state our boundedness theorem in the Bergman-Orlicz spaces framework
as follows:
\begin{stthm}
\label{thm|Embed_Thm_Berg_Cont}Let $\mu$ be a finite positive Borel
measure on $\mathbb{B}_{N}$ and let $\psi_{1}$ and $\psi_{2}$ be
two Orlicz functions. Then:
\begin{enumerate}
\item If inclusion $A_{\alpha}^{\psi_{1}}\left(\mathbb{B}_{N}\right)\subset L^{\psi_{2}}\left(\mu\right)$
holds and is continuous, then there exists some $A>0$ such that\begin{equation}
\varrho_{\mu}\left(h\right)=O_{h\rightarrow0}\left(\frac{1}{\psi_{2}\left(A\psi_{1}^{-1}\left(1/h^{N+1+\alpha}\right)\right)}\right).\label{eq|R_Berg_Orlicz}\end{equation}

\item If there exists some $A>0$ such that\begin{equation}
K_{\mu,\alpha}\left(h\right)=O_{h\rightarrow0}\left(\frac{1/h^{N+1+\alpha}}{\psi_{2}\left(A\psi_{1}^{-1}\left(1/h^{N+1+\alpha}\right)\right)}\right)\label{eq|K_Berg_Orlicz}\end{equation}
then inclusion $A_{\alpha}^{\psi_{1}}\left(\mathbb{B}_{N}\right)\subset L^{\psi_{2}}\left(\mu\right)$
holds and is continuous.
\item If in addition $\psi_{1}=\psi_{2}=\psi$ satisfies the uniform $\nabla_{0}$-Condition,
then Conditions \prettyref{eq|R_Berg_Orlicz} and \prettyref{eq|K_Berg_Orlicz}
are equivalent.
\end{enumerate}
\end{stthm}
Note that embedding $A_{\alpha}^{\psi_{1}}\left(\mathbb{B}_{N}\right)\subset L^{\psi_{2}}\left(\mu\right)$
is continuous as soon as it holds. It is just an application of the
closed graph theorem.
\begin{proof}
[Proof of \prettyref{thm|Embed_Thm_Berg_Cont}]1) For the first
part, let us denote by $C$ the norm of the canonical embedding $j_{\alpha}:A_{\alpha}^{\psi_{1}}\left(\mathbb{B}_{N}\right)\hookrightarrow L^{\psi_{2}}\left(\mu\right)$.
Let $a\in\mathbb{B}_{N}$, $\left|a\right|=1-h$ and $\xi\in\mathbb{S}_{N}$
be such that $a=\left(1-h\right)\xi$. Let us consider the map\begin{eqnarray*}
f_{a} & = & \frac{1}{2^{N+1+\alpha}}\frac{\psi_{1}^{-1}\left(1/h^{N+1+\alpha}\right)}{1/h^{N+1+\alpha}}H_{a}\left(z\right)\\
 & = & \frac{1}{2^{N+1+\alpha}}\frac{\psi_{1}^{-1}\left(1/h^{N+1+\alpha}\right)}{1/h^{N+1+\alpha}}\left(\frac{h\left(2-h\right)}{\left|1-\left(1-h\right)\left\langle z,\xi\right\rangle \right|^{2}}\right)^{N+1+\alpha}\end{eqnarray*}
Recall that $H_{a}$ is the Berezin kernel introduced in \prettyref{prop|point_eval_func_Bergman}.
As we saw in the proof of this latter, $f_{a}$ is in the unit ball
of $A_{\alpha}^{\psi_{1}}\left(\mathbb{B}_{N}\right)$ and our assumption
ensures that\[
\left\Vert j_{\alpha}\left(f_{a}\right)\right\Vert _{L^{\psi_{2}}\left(\mu\right)}=\left\Vert f_{a}\right\Vert _{L^{\psi_{2}}\left(\mu\right)}\leq C\]
so that\begin{equation}
1\geq\int_{\mathbb{B}_{N}}\psi_{2}\left(\frac{\left|f_{a}\right|}{C}\right)d\mu.\label{eq|Berg_Orlicz_Th1_sens_1_estimate}\end{equation}
Let us minorize the right hand side of \prettyref{eq|Berg_Orlicz_Th1_sens_1_estimate}.
We just get a minorization of $\left|f_{a}\right|$ on the non-isotropic
{}``ball'' $S\left(\xi,h\right)$. If $z\in S\left(\xi,h\right)$,
then a straightforward computation yields $\left|1-\left\langle z,a\right\rangle \right|\leq2h$.
Hence, for any $z\in S\left(a,h\right)$, \[
{\displaystyle \left|f_{a}\left(z\right)\right|}\geq\frac{\psi_{1}^{-1}\left(1/h^{N+1+\alpha}\right)}{8^{N+1+\alpha}}.\]
 Therefore\[
1\geq\int_{\mathbb{B}_{N}}\psi_{2}\left(\frac{\left|f\right|}{C}\right)d\mu\geq\psi_{2}\left(\frac{\psi_{1}^{-1}\left(1/h^{N+1+\alpha}\right)}{8^{N+1+\alpha}C}\right)\mu\left(S\left(a,h\right)\right),\]
which is Condition \prettyref{eq|R_Berg_Orlicz} and the first part
of the theorem follows.

2) The second part will need \prettyref{lem|lem_tech_Berg_Orl}. First
of all, we know (\prettyref{prop|cont_ma_funct_Berg_Orlicz}) that
there exists a constant $C_{M}\geq1$ such that, for every $f\in A_{\alpha}^{\psi_{1}}\left(\mathbb{B}_{N}\right)$,
$\left\Vert \Lambda_{f}\right\Vert _{L_{\alpha}^{\psi_{1}}\left(\mathbb{B}_{N}\right)}\leq C_{M}\left\Vert f\right\Vert _{A_{\alpha}^{\psi_{1}}\left(\mathbb{B}_{N}\right)}$.
Let now $f$ be in the unit ball of $A_{\alpha}^{\psi_{1}}\left(\mathbb{B}_{N}\right)$;
it suffices to show that $\left\Vert f\right\Vert _{L^{\psi_{2}}\left(\mu\right)}\leq C_{0}$
for some constant $C_{0}>0$ which does not depend on $f$. Let $\tilde{C}\geq1$
be a constant whose value will be precised later. Condition \prettyref{eq|K_Berg_Orlicz}
is supposed to be realized, that is there exist some constants $A>0$,
$h_{A}\in\left(0,1/2\right]$ and $\eta>0$ such that\begin{equation}
K_{\mu,\alpha}\left(h\right)\leq\eta\frac{1/h^{N+1+\alpha}}{\psi_{2}\left(A\psi_{1}^{-1}\left(1/h^{N+1+\alpha}\right)\right)}\label{eq|Berg_Thm_cont_emb_second_point_assumed}\end{equation}
for any $h\in\left(0,h_{A}\right)$. By using convexity of $\psi_{2}$
and applying \prettyref{lem|lem_tech_Berg_Orl} to $f/C_{M}$ (which
of course still satisfies $\left\Vert f/C_{M}\right\Vert _{A_{\alpha}^{\psi_{1}}}\leq1$)
and $E=\mathbb{B}_{N}$, we get the existence of constants $B>0$,
$x_{A}$ and $C_{1}>0$, all independent of $f$, such that\begin{eqnarray*}
\int_{\mathbb{B}_{N}}\psi_{2}\left(\frac{\left|f\right|}{BC_{M}\tilde{C}}\right)d\mu & \leq & \frac{1}{\tilde{C}}\int_{\mathbb{B}_{N}}\psi_{2}\left(\frac{\left|f\right|}{BC_{M}}\right)d\mu\\
 & \leq & \frac{1}{\tilde{C}}\left(\mu\left(\mathbb{B}_{N}\right)\psi_{2}\left(x_{A}\right)+C_{1}\eta\right).\end{eqnarray*}
Of course, $C_{1}$ may be supposed to be large enough so that $C_{1}\eta\geq1$
and, up to fix $\tilde{C}=\mu\left(\mathbb{B}_{N}\right)\psi_{2}\left(x_{A}\right)+C_{1}\eta\geq1$,
we get ${\displaystyle \left\Vert f\right\Vert _{L^{\psi_{2}}\left(\mu\right)}\leq C_{0}:=BC_{M}\tilde{C}}$
which completes the proof of (2) of \prettyref{thm|Embed_Thm_Berg_Cont}.

3) First, it is clear that Condition \prettyref{eq|K_Berg_Orlicz}
implies Condition \prettyref{eq|R_Berg_Orlicz}. For the converse,
we need the following claim:
\begin{claim*}
Under the notations of the theorem, if Condition \prettyref{eq|R_Berg_Orlicz}
holds, then there exist some $A$ as large as we want and $\eta>0$
such that\begin{equation}
\varrho_{\mu}\left(h\right)\leq\eta\frac{1}{\psi_{2}\left(A\psi_{1}^{-1}\left(h_{A}/h^{N+1+\alpha}\right)\right)}\label{eq|remark_R_refined}\end{equation}
for some $h_{A}$, $0<h_{A}\leq1$ and for any $0<h<h_{A}$.\end{claim*}
\begin{proof}
[Proof of the claim]We assume that Condition\begin{equation}
\varrho_{\mu}\left(h\right)\leq\eta\frac{1}{\psi_{2}\left(\tilde{A}\psi_{1}^{-1}\left(1/h^{N+1+\alpha}\right)\right)}\label{eq|proof_3_Berg_Cont_Emb}\end{equation}
holds for some $\tilde{A}\geq0$, $\tilde{h_{A}}$, $0<\tilde{h_{A}}\leq1$,
$\eta>0$ and any $0<h<\tilde{h_{A}}$. We fix $A>1$ and we look
for some constant $h_{\tilde{A},A}\leq1$ such that\begin{equation}
\frac{1}{\psi_{2}\left(\tilde{A}\psi_{1}^{-1}\left(1/h^{N+1+\alpha}\right)\right)}\leq\frac{1}{\psi_{2}\left(A\psi_{1}^{-1}\left(\left(h_{\tilde{A},A}/h\right)^{N+1+\alpha}\right)\right)}\label{eq|R_Berg_refined-bis}\end{equation}
for $0<h<h_{\tilde{A},A}$. Now it is easy to verify that Inequality
\prettyref{eq|R_Berg_refined-bis} is equivalent to \[
\frac{A}{\tilde{A}}\leq\frac{\psi_{1}^{-1}\left(1/h^{N+1+\alpha}\right)}{\psi_{1}^{-1}\left(\left(h_{\tilde{A},A}/h\right)^{N+1+\alpha}\right)}\leq\frac{1}{h_{\tilde{A},A}^{N+1+\alpha}}\]
by concavity of $\psi^{-1}$. Then the claim follows by choosing $h_{\tilde{A},A}$
small enough.
\end{proof}
We come back to the proof of the third point. Let suppose that $\psi$
belongs to the uniform $\nabla_{0}$-class and let $A>0$, $h_{A}\in\left(0,1\right]$
and $\eta>0$ be such that\[
\varrho_{\mu}\left(h\right)\leq\eta\frac{1}{\psi\left(A\psi^{-1}\left(1/h^{N+1+\alpha}\right)\right)}\]
for every $h\in\left(0,h_{A}\right)$. The previous claim says that
we can find $B\geq1$ and $0<K=K_{B,A}\leq1$ such that \[
\varrho_{\mu}\left(h\right)\leq\eta\frac{1}{\psi\left(B\psi^{-1}\left(\left(K/h\right)^{N+1+\alpha}\right)\right)}\]
for every $0<h<K$. Therefore, we have\begin{multline*}
K_{\mu,\alpha}\left(h\right)=\sup_{0<t\leq h}\frac{\varrho_{\mu}\left(t\right)}{t^{N+1+\alpha}}\leq\eta\sup_{0<t\leq h}\frac{1/t^{N+1+\alpha}}{\psi\left(B\psi^{-1}\left(\left(K/t\right)^{N+1+\alpha}\right)\right)}\\
=\eta\sup_{x\geq\psi^{-1}\left(\left(K/h\right)^{N+1+\alpha}\right)}\frac{1}{K^{N+1+\alpha}}\frac{\psi\left(x\right)}{\psi\left(Bx\right)}\end{multline*}
for any $0<h\leq K$. Let $C$ be the constant induced by the uniform
$\nabla_{0}$-Condition satisfied by $\psi$ and let $\beta$ be such
that $B=\beta C$. The claim allows us to take $B$ large enough and
therefore to assume that $\beta>1$. We then have, since $\psi$ satisfies
the uniform $\nabla_{0}$-Condition,\[
\frac{\psi\left(\beta\psi^{-1}\left(\left(K/h\right)^{N+1+\alpha}\right)\right)}{\left(K/h\right)^{N+1+\alpha}}\leq\frac{\psi\left(Bx\right)}{\psi\left(x\right)}\]
for any $x\geq\psi^{-1}\left(\left(K/h\right)^{N+1+\alpha}\right)$.
Hence, for every $0<h\leq K$,\[
K_{\mu,\alpha}\left(h\right)\leq\eta\frac{1/h^{N+1+\alpha}}{\psi\left(\beta\psi^{-1}\left(\left(K/h\right)^{N+1+\alpha}\right)\right)}\leq\eta\frac{1/h^{N+1+\alpha}}{\psi\left(\beta K^{N+1+\alpha}\psi^{-1}\left(1/h^{N+1+\alpha}\right)\right)}\]
by concavity of $\psi^{-1}$, and Condition \prettyref{eq|K_Berg_Orlicz}
is satisfied.
\end{proof}
\smallskip{}
The third point of the previous theorem leads us to define $\left(\psi,\alpha\right)$\textit{-Bergman-Carleson}
measures on the ball:
\begin{stdefn}
Let $\mu$ be a positive Borel measure on $\mathbb{B}_{N}$ and let
$\psi$ be an Orlicz function. We say that $\mu$ is a $\left(\psi,\alpha\right)$-Bergman-Carleson
measure if there exists some $A>0$, such that \begin{equation}
\mu\left(W\left(\xi,h\right)\right)=O_{h\rightarrow0}\left(\frac{1}{\psi\left(A\psi^{-1}\left(1/h^{N+1+\alpha}\right)\right)}\right)\label{eq|def_Berg_Carle_meas}\end{equation}
uniformly with respect to $\xi\in\mathbb{S}_{N}$.
\end{stdefn}
We notice that \prettyref{eq|def_Berg_Carle_meas} is equivalent to
\prettyref{eq|R_Berg_Orlicz}. Therefore, we can state the following
corollary:
\begin{stcor}
\label{cor|Cor_Berg_nabla_cont}Let $\mu$ be a finite positive Borel
measure on $\mathbb{B}_{N}$ and let $\psi$ be an Orlicz function
satisfying the uniform $\nabla_{0}$-Condition. Inclusion $A_{\alpha}^{\psi}\left(\mathbb{B}_{N}\right)\hookrightarrow L^{\psi}\left(\mu\right)$
holds (and is continuous) if and only if $\mu$ is a $\left(\psi,\alpha\right)$-Bergman-Carleson
measure.
\end{stcor}

\subsection{\label{subsub|Berg_Comp_emb}Compactness of the canonical embedding
$A_{\alpha}^{\psi_{1}}\left(\mathbb{B}_{N}\right)\hookrightarrow L^{\psi_{2}}\left(\mu\right)$.}

For the study of compactness, we usually need some compactness criterion.
\begin{stprop}
\label{prop|cor_crit_comp_Berg}Let $\mu$ be a finite positive measure
on $\mathbb{B}_{N}$ and let $\psi_{1}$ and $\psi_{2}$ be two Orlicz
functions. We suppose that the canonical embedding $j_{\mu,\alpha}:A_{\alpha}^{\psi_{1}}\left(\mathbb{B}_{N}\right)\hookrightarrow L^{\psi_{2}}\left(\mu\right)$
holds and is bounded. The three following assertions are equivalent:
\begin{enumerate}
\item $j_{\mu,\alpha}:A_{\alpha}^{\psi_{1}}\left(\mathbb{B}_{N}\right)\hookrightarrow L^{\psi_{2}}\left(\mu\right)$
is compact;
\item Every sequence in the unit ball of $A_{\alpha}^{\psi_{1}}\left(\mathbb{B}_{N}\right)$,
which is convergent to $0$ uniformly on every compact subset of $\mathbb{B}_{N}$,
is strongly convergent to $0$ in $L^{\psi_{2}}\left(\mu\right)$.
\item $\lim_{r\rightarrow1^{-}}\left\Vert I_{r}\right\Vert =0$, where $I_{r}\left(f\right)=f.\chi_{\mathbb{B}_{N}\setminus r\overline{\mathbb{B}_{N}}}$.
\end{enumerate}
\end{stprop}
\begin{proof}
(1)$\,\Rightarrow\,$(2) We first assume that $j_{\mu,\alpha}$ is
compact. Let $\left(f_{n}\right)_{n}$ be a sequence in the unit ball
of $A_{\alpha}^{\psi_{1}}\left(\mathbb{B}_{N}\right)$, which is convergent
to $0$ uniformly on every compact subset of $\mathbb{B}_{N}$. Of
course, $j_{\mu,\alpha}\left(f_{n}\right)$ converges to $0$ everywhere.
By contradiction, suppose up to extract a subsequence that $\liminf_{n}\left\Vert j_{\mu,\alpha}\left(f_{n}\right)\right\Vert _{L^{\psi_{2}}\left(\mu\right)}>0$.
By compactness of $j_{\mu,\alpha}$, up to an other extraction, we
may assume that $\left(j_{\mu,\alpha}\left(f_{n}\right)\right)_{n}$
strongly converges to some $g\in L^{\psi_{2}}\left(\mu\right)$ and
we must have $\left\Vert g\right\Vert _{L^{\psi_{2}}\left(\mu\right)}>0$.
As convergence in norm in $L^{\psi_{2}}\left(\mu\right)$ entails
$\mu$-almost everywhere convergence, we get a contradiction.

(2)$\,\Rightarrow\,$(1) Conversely, let $\left(f_{n}\right)_{n}$
be a sequence in the unit ball of $A_{\alpha}^{\psi_{1}}\left(\mathbb{B}_{N}\right)$.
In particular, $\left(f_{n}\right)_{n}$ is in the unit ball of $A_{\alpha}^{1}\left(\mathbb{B}_{N}\right)$
and the Cauchy's formula ensures that $\left(f_{n}\right)_{n}$ is
uniformly bounded on every compact subset of $\mathbb{B}_{N}$, so
that, up to an extraction, we may suppose that $\left(f_{n}\right)_{n}$
is uniformly convergent on compact subsets of $\mathbb{B}_{N}$ to
$f$ holomorphic in $\mathbb{B}_{N}$, by Montel's theorem. Now, Lebesgue's
theorem ensures that $f\in A_{\alpha}^{\psi_{1}}\left(\mathbb{B}_{N}\right)$
and, up to divide by a constant large enough, we may assume that $f_{n}-f$,
which converges to $0$ on every compact subset of $\mathbb{B}_{N}$,
is in the unit ball of $A_{\alpha}^{\psi_{1}}\left(\mathbb{B}_{N}\right)$.
Therefore, our assumption implies that $\left(j_{\mu,\alpha}\left(f_{n}\right)-j_{\mu,\alpha}\left(f\right)\right)_{n}$
converges to $0$ in the norm of $L^{\psi_{2}}\left(\mu\right)$ and
$j_{\mu,\alpha}$ is compact, as expected.

(3)$\,\Rightarrow\,$(2) Let $\left(f_{n}\right)_{n}$ be in the unit
ball of $A_{\alpha}^{\psi_{1}}\left(\mathbb{B}_{N}\right)$ converging
to $0$ uniformly on every compact subset of $\mathbb{B}_{N}$. We
have\begin{eqnarray*}
\limsup_{n\rightarrow\infty}\left\Vert f_{n}\right\Vert _{L^{\psi_{2}}\left(\mu\right)} & = & \limsup_{r\rightarrow1^{-}}\limsup_{n\rightarrow\infty}\left\Vert I_{r}\left(f_{n}\right)+f_{n}.\chi_{\overline{r\mathbb{B}_{N}}}\right\Vert _{L^{\psi_{2}}\left(\mu\right)}\\
 & \lesssim & \limsup_{r\rightarrow1^{-}}\left\Vert I_{r}\right\Vert +\limsup_{r\rightarrow1^{-}}\limsup_{n\rightarrow\infty}\left\Vert f_{n}.\chi_{\overline{r\mathbb{B}_{N}}}\right\Vert _{\infty}\\
 & = & 0.\end{eqnarray*}

(2)$\,\Rightarrow\,$(3) By contradiction suppose that (3) is not
satisfied so that there exist a constant $\delta>0$ and a sequence
$\left(f_{n}\right)_{n}$ in the unit ball of $A_{\alpha}^{\psi_{1}}\left(\mathbb{B}_{N}\right)$
such that ${\displaystyle \left\Vert I_{\left(1-\frac{1}{n}\right)}\left(f_{n}\right)\right\Vert }_{L^{\psi_{2}}}\geq\delta$,
for every $n\geq0$. Up to an extraction, we may suppose that $\left(f_{n}\right)_{n}$
converges uniformly on compact subsets of $\mathbb{B}_{N}$ to $f\in A_{\alpha}^{\psi_{1}}\left(\mathbb{B}_{N}\right)$.
By Lebesgue's theorem, $\lim_{n\rightarrow\infty}\left\Vert I_{\left(1-\frac{1}{n}\right)}\left(f\right)\right\Vert _{L^{\psi_{2}}}=0$;
thus, for $n$ large enough,\[
\left\Vert f_{n}-f\right\Vert _{L^{\psi_{2}}}\geq\left\Vert I_{\left(1-\frac{1}{n}\right)}\left(f_{n}-f\right)\right\Vert _{L^{\psi_{2}}}\geq\delta/2\]
which contradicts (2).
\end{proof}
As for the boundedness, we state our embedding compactness theorem
for weighted Bergman-Orlicz spaces as follows:
\begin{stthm}
\label{thm|Embed_Berg_orlizc_Comp}Let $\mu$ be a finite positive
Borel measure on $\mathbb{B}_{N}$, and let $\psi_{1}$ and $\psi_{2}$
be two Orlicz functions.
\begin{enumerate}
\item If the inclusion $A_{\alpha}^{\psi_{1}}\left(\mathbb{B}_{N}\right)\subset L^{\psi_{2}}\left(\mu\right)$
holds and is compact, then for every $A>0$ we have\begin{equation}
\varrho_{\mu}\left(h\right)=o_{h\rightarrow0}\left(\frac{1}{\psi_{2}\left(A\psi_{1}^{-1}\left(1/h^{N+1+\alpha}\right)\right)}\right).\label{eq|R_0_Berg_Orlicz}\end{equation}

\item If\begin{equation}
K_{\mu,\alpha}\left(h\right)=o_{h\rightarrow0}\left(\frac{1/h^{N+1+\alpha}}{\psi_{2}\left(A\psi_{1}^{-1}\left(1/h^{N+1+\alpha}\right)\right)}\right)\label{eq|K_0_Berg_Orlicz}\end{equation}
for every $A>0$, then $A_{\alpha}^{\psi_{1}}\left(\mathbb{B}_{N}\right)$
embeds compactly in $L^{\psi_{2}}\left(\mu\right)$.
\item If in addition $\psi_{1}=\psi_{2}=\psi$ satisfies the $\nabla_{0}$-Condition,
then Conditions \prettyref{eq|R_0_Berg_Orlicz} and \prettyref{eq|K_0_Berg_Orlicz}
are equivalent.
\end{enumerate}
\end{stthm}
\begin{proof}
1) We suppose that the canonical embedding is compact but that Condition
\prettyref{eq|R_0_Berg_Orlicz} failed to be satisfied. This means
that there exist some $\varepsilon_{0}\in\left(0,1\right)$ and $A>0$,
some sequences $\left(h_{n}\right)_{n}\subset\left(0,1\right)$ decreasing
to $0$ and $\left(\xi_{n}\right)_{n}\subset\mathbb{S}_{N}$, such
that\[
\mu\left(S\left(\xi_{n},h_{n}\right)\right)\geq\frac{\varepsilon_{0}}{\psi_{2}\left(A\psi_{1}^{-1}\left(1/h^{N+1+\alpha}\right)\right)}.\]
Let $a_{n}:=\left(1-h_{n}\right)\xi_{n}$ and consider the functions\begin{eqnarray}
f_{n}\left(z\right):=f_{a_{n}}\left(z\right) & := & \frac{1}{2^{N+1+\alpha}}\frac{\psi_{1}^{-1}\left(1/h_{n}^{N+1+\alpha}\right)}{1/h_{n}^{N+1+\alpha}}H_{a_{n}}\left(z\right)\label{eq|etiq_calc_comp_Bergm}\end{eqnarray}
where $H_{a_{n}}$ is the Berezin kernel, as in the proof of the first
part of \prettyref{thm|Embed_Thm_Berg_Cont}. Every $f_{n}$ lays
in the unit ball of $A_{\alpha}^{\psi_{1}}\left(\mathbb{B}_{N}\right)$
and $\left(f_{n}\right)_{n}\xrightarrow[n\rightarrow\infty]{}0$ uniformly
on every compact subset of $\mathbb{B}_{N}$. So \prettyref{prop|cor_crit_comp_Berg}
ensures that $\left(f_{n}\right)_{n}$ converges to $0$ in norm of
$L^{\psi_{2}}\left(\mu\right)$.

Now, by the proof of the first part of \prettyref{thm|Embed_Thm_Berg_Cont},
the following estimation holds:\[
\left|f_{n}\left(z\right)\right|\geq\frac{\psi_{1}^{-1}\left(1/h_{n}^{N+1+\alpha}\right)}{8^{N+1+\alpha}}\]
for any $z\in S\left(\xi_{n},h_{n}\right)$; therefore\begin{eqnarray*}
\int_{\mathbb{B}_{N}}\psi_{2}\left(\frac{8^{N+1+\alpha}A}{\varepsilon_{0}}\left|f_{n}\right|\right)d\mu & \geq & \psi_{2}\left(\frac{A}{\varepsilon_{0}}\psi_{1}^{-1}\left(\frac{1}{h_{n}^{N+1+\alpha}}\right)\right)\mu\left(S\left(\xi_{n},h_{n}\right)\right)\\
 & \geq & \psi_{2}\left(\frac{A}{\varepsilon_{0}}\psi_{1}^{-1}\left(\frac{1}{h_{n}^{N+1+\alpha}}\right)\right)\frac{\varepsilon_{0}}{\psi_{2}\left(A\psi_{1}^{-1}\left(1/h_{n}^{N+1+\alpha}\right)\right)}\\
 & \geq & 1\end{eqnarray*}
by the convexity of $\psi_{2}$. This yields ${\displaystyle \left\Vert f_{n}\right\Vert _{L^{\psi_{2}}\left(\mu\right)}\geq\frac{\varepsilon_{0}}{8^{N+1+\alpha}A}}$
for every $n$, which is a contradiction and gives the first part.

2) We now assume that Condition \prettyref{eq|K_0_Berg_Orlicz} is
satisfied. Thanks to the second point of \prettyref{prop|cor_crit_comp_Berg},
it is sufficient to prove that, for every $\varepsilon>0$, the norm
of the embedding\[
I_{r}:A_{\alpha}^{\psi_{1}}\left(\mathbb{B}_{N}\right)\hookrightarrow L^{\psi_{2}}\left(\mathbb{B}_{N}\setminus r\overline{\mathbb{B}_{N}},\mu\right)\]
is smaller than $\varepsilon$ for some $r_{0}\left(\varepsilon\right)$
and every $r$ such that $r_{0}\left(\varepsilon\right)\leq r<1$.
Let $\eta\in\left(0,1\right)$ and let ${\displaystyle A:=A\left(\varepsilon\right)=\frac{6.4^{N+\alpha}}{\varepsilon}>0}$;
Condition \prettyref{eq|K_0_Berg_Orlicz} ensures that there exists
$h_{A}\in\left(0,1/2\right)$ such that\[
K_{\mu,\alpha}\left(h\right)\leq\eta\frac{1/h^{N+1+\alpha}}{\psi_{2}\left(A\psi_{1}^{-1}\left(1/h^{N+1+\alpha}\right)\right)}\]
for $h\leq h_{A}$. Let now $f$ be in the unit ball of $A_{\alpha}^{\psi_{1}}\left(\mathbb{B}_{N}\right)$
and $r\in\left(0,1\right)$. By the proof of \prettyref{lem|lem_tech_Berg_Orl},
applied to $E=\mathbb{B}_{N}\setminus r\overline{\mathbb{B}_{N}}$
and $f$, there exist a constant $B>0$ given by $B={\displaystyle \frac{6.4^{N+\alpha}}{A}}=\varepsilon$,
and some constants $x_{A}>0$ and $C_{1}>0$, independent of $f$,
such that\begin{eqnarray*}
\int_{\mathbb{B}_{N}\setminus r\overline{\mathbb{B}_{N}}}\psi_{2}\left(\frac{\left|f\right|}{\varepsilon}\right)d\mu & = & \int_{\mathbb{B}_{N}\setminus r\overline{\mathbb{B}_{N}}}\psi_{2}\left(\frac{\left|f\right|}{B}\right)d\mu\\
 & \leq & \mu\left(\mathbb{B}_{N}\setminus r\overline{\mathbb{B}_{N}}\right)\psi_{2}\left(x_{A}\right)+C_{1}\eta\int_{\mathbb{B}_{N}}\psi_{1}\left(\Lambda_{f}\right)dv_{\alpha}.\end{eqnarray*}
Now, we choose $\eta$ such that ${\displaystyle C_{1}\eta\int_{\mathbb{B}_{N}}\psi_{1}\left(\Lambda_{f}\right)dv_{\alpha}\leq\frac{1}{2}}$
(which is possible thanks to \prettyref{prop|cont_ma_funct_Berg_Orlicz})
and we take $r_{0}\in\left(0,1\right)$ such that ${\displaystyle \mu\left(\mathbb{B}_{N}\setminus r\overline{\mathbb{B}_{N}}\right)\psi_{2}\left(x_{A}\right)\leq\frac{1}{2}}$
for every $r\in\left(r_{0},1\right)$. We get $\left\Vert I_{r}\left(f\right)\right\Vert _{L^{\psi_{2}}\left(\mu\right)}\leq\varepsilon$
as soon as $r_{0}<r<1$, what completes the proof.

3) The proof of the third point is essentially contained in that of
the third part of \cite[Theorem 4.11]{QUEF-LI-LE-RO-PI}.
\end{proof}
\smallskip{}
This leads us to the definition of \textit{vanishing} $\left(\psi,\alpha\right)$\textit{-Bergman-Carleson}
measures on the ball:
\begin{stdefn}
Let $\psi$ be an Orlicz function and let $\mu$ be a Borel positive
measure on $\mathbb{B}_{N}$. We say that $\mu$ is a \textit{vanishing}
$\left(\psi,\alpha\right)$\textit{-Bergman-Carleson} measure if,
for every $A>0$,\[
\mu\left(W\left(\xi,h\right)\right)=o_{h\rightarrow0}\left(\frac{1}{\psi\left(A\psi^{-1}\left(1/h^{N+1+\alpha}\right)\right)}\right)\]
uniformly with respect to $\xi\in\mathbb{S}_{N}$.
\end{stdefn}
We have the following corollary:
\begin{stcor}
\label{cor|cor_Berg_nabla_comp}Let $\psi$ be an Orlicz function
satisfying the $\nabla_{0}$-Condition and let $\mu$ be a Borel positive
measure on $\mathbb{B}_{N}$. Then $A_{\alpha}^{\psi}\left(\mathbb{B}_{N}\right)$
embeds compactly into $L^{\psi}\left(\mu\right)$ if and only if $\mu$
is a vanishing $\left(\psi,\alpha\right)$-Bergman-Carleson measure.
\end{stcor}

\section{\label{sec|B_O_Application-to-composition}Application to composition
operators on weighted Bergman-Orlicz spaces.}

For $\phi:\mathbb{B}_{N}\rightarrow\mathbb{B}_{N}$ analytic, we denote
by $\mu_{\phi}^{\alpha}$ the pull-back measure by $\phi$ of the
weighted Lebesgue measure $v_{\alpha}$ on $\mathbb{B}_{N}$, namely
$\mu_{\phi}^{\alpha}\left(E\right)=v_{\alpha}\left(\phi^{-1}\left(E\right)\right)$
for every Borel subset $E$ of $\mathbb{B}_{N}$.

\prettyref{thm|Embed_Thm_Berg_Cont} and \prettyref{thm|Embed_Berg_orlizc_Comp}
allow us to give the following characterization with some constraints
on the Orlicz function $\psi$:
\begin{stthm}
\label{thm|caract_cont_comp_comp_op_Berg}Let $\psi$ be an Orlicz
function and let $\phi:\mathbb{B}_{N}\rightarrow\mathbb{B}_{N}$ be
holomorphic.
\begin{enumerate}
\item If $\psi$ satisfies the uniform $\nabla_{0}$-Condition, then $C_{\phi}$
is bounded from $A_{\alpha}^{\psi}\left(\mathbb{B}_{N}\right)$ into
itself if and only if $\mu_{\phi}^{\alpha}$ is a $\left(\psi,\alpha\right)$-Bergman-Carleson
measure.
\item If $\psi$ satisfies the $\nabla_{0}$-Condition, then $C_{\phi}$
is compact from $A_{\alpha}^{\psi}\left(\mathbb{B}_{N}\right)$ into
itself if and only if $\mu_{\phi}^{\alpha}$ is a vanishing $\left(\psi,\alpha\right)$-Bergman-Carleson
measure.
\end{enumerate}
\end{stthm}
\begin{proof}
Thanks to \prettyref{cor|Cor_Berg_nabla_cont} and \prettyref{cor|cor_Berg_nabla_comp},
it suffices to notice that the continuity (resp. compactness) of the
canonical embedding $j_{\mu_{\phi}^{\alpha}}:A_{\alpha}^{\psi}\left(\mathbb{B}_{N}\right)\hookrightarrow L^{\psi}\left(\mu_{\phi}^{\alpha}\right)$
is equivalent to the boundedness (resp. compactness) of $C_{\phi}:A_{\alpha}^{\psi}\left(\mathbb{B}_{N}\right)\rightarrow A_{\alpha}^{\psi}\left(\mathbb{B}_{N}\right)$.
This just proceeds from the fact that\begin{eqnarray*}
\left\Vert C_{\phi}\left(f\right)\right\Vert _{A_{\alpha}^{\psi}\left(\mathbb{B}_{N}\right)} & = & \inf\left\{ C>0,\,\int_{\mathbb{B}_{N}}\psi\left(\frac{\left|f\circ\phi\right|}{C}\right)d\sigma\leq1\right\} \\
 & = & \inf\left\{ C>0,\,\int_{\mathbb{B}_{N}}\psi\left(\frac{\left|f\right|}{C}\right)d\mu_{\phi}\leq1\right\} \\
 & = & \left\Vert j_{\mu_{\phi}^{\alpha}}\left(f\right)\right\Vert _{L^{\psi}\left(\mu_{\phi}^{\alpha}\right)},\end{eqnarray*}
for any $f\in A_{\alpha}^{\psi}\left(\mathbb{B}_{N}\right)$.\end{proof}
\begin{strem}
If we do not assume that $\psi$ satisfies the uniform $\nabla_{0}$-Condition
(resp. $\nabla_{0}$-Condition), then \prettyref{thm|Embed_Thm_Berg_Cont}
(resp. \prettyref{thm|Embed_Berg_orlizc_Comp}) provides \textit{a
priori} non-equivalent necessary and sufficient conditions to the
boundedness (resp. compactness) of $C_{\phi}$ on $A_{\alpha}^{\psi}\left(\mathbb{B}_{N}\right)$.
\end{strem}
As a particular case of the previous theorem, we state and verify
\cite[Theorem 3.6 and Theorem 4.3]{JIANG}:
\begin{stthm}
\label{thm|cont_comp_comp_op_delta_2_Berg}Let $\psi$ be an Orlicz
function which satisfies $\Delta_{2}$-Conditions and let $\phi:\mathbb{B}_{N}\rightarrow\mathbb{B}_{N}$
be holomorphic. Then
\begin{enumerate}
\item $C_{\phi}$ is bounded from $A_{\alpha}^{\psi}\left(\mathbb{B}_{N}\right)$
into itself if and only if $\mu_{\phi}^{\alpha}$ is an $\alpha$-Bergman-Carleson
measure.
\item $C_{\phi}$ is compact from $A_{\alpha}^{\psi}\left(\mathbb{B}_{N}\right)$
into itself if and only if $\mu_{\phi}^{\alpha}$ is a vanishing $\alpha$-Bergman-Carleson
measure.
\end{enumerate}
\end{stthm}
\begin{proof}
It suffices to observe that\[
\frac{1}{\psi\left(A\psi^{-1}\left(1/h^{N+1+\alpha}\right)\right)}\approx h^{N+1+\alpha}\]
for every $A>0$, whenever $\psi$ is an Orlicz function which satisfies
the $\Delta_{2}$-Condition (see Remark 2 (a) following Theorem 4.11
in \cite{QUEF-LI-LE-RO-PI}.)
\end{proof}
\smallskip{}

A first consequence of these characterizations is the following:
\begin{stcor}
Let $\phi:\mathbb{B}_{N}\rightarrow\mathbb{B}_{N}$ be holomorphic
and let $\psi,\nu$ be two Orlicz functions. Assume that $\nu$ satisfies
the $\Delta_{2}$-Condition. Then:
\begin{enumerate}
\item If $C_{\phi}$ is bounded on $A_{\alpha}^{\nu}\left(\mathbb{B}_{N}\right)$
(e.g. on any $A_{\alpha}^{p}\left(\mathbb{B}_{N}\right)$), then is
it bounded on $A_{\alpha}^{\psi}\left(\mathbb{B}_{N}\right)$;
\item If $\psi$ satisfies the $\nabla_{0}$-Condition and if $C_{\phi}$
is compact on $A_{\alpha}^{\psi}\left(\mathbb{B}_{N}\right)$, then
it is compact on $A_{\alpha}^{\nu}\left(\mathbb{B}_{N}\right)$ (e.g.
on any $A_{\alpha}^{p}\left(\mathbb{B}_{N}\right)$).
\end{enumerate}
\end{stcor}
\begin{proof}
The first point follows from the remark before \prettyref{thm|cont_comp_comp_op_delta_2_Berg}
and from the fact that if $\mu$ is a $\alpha$-Carleson measure,
i.e. if $K_{\mu,\alpha}\leq C$ for some constant $C\geq1$, then
$\mu$ is a $\left(\psi,\alpha\right)$-Carleson measure, since $\psi\left(A\psi^{-1}\left(1/h^{N+1+\alpha}\right)\right)\leq A/h^{N+1+\alpha}$,
for any $0<A\leq1$.

For the second point, it suffices to show that Condition \prettyref{eq|K_0_Berg_Orlicz}
implies that $\mu$ is a vanishing $\alpha$-Carleson measure, what
is trivial if we apply it for $A=1$.
\end{proof}
\medskip{}
As one of the main motivation to this work, we are interested in finding
where the break of condition for boundedness of $C_{\phi}$ happens
between $H^{\infty}\left(\mathbb{B}_{N}\right)$ and $A_{\alpha}^{p}\left(\mathbb{B}_{N}\right)$.
More precisely, we wonder if there are some spaces different from
$H^{\infty}\left(\mathbb{B}_{N}\right)$ and smaller than some $A_{\alpha}^{p}\left(\mathbb{B}_{N}\right)$
on which every composition operator $C_{\phi}$ is bounded. In \cite{MACCLUER-MERCER},
the authors show the following proposition:
\begin{stprop}
\label{prop|MacCluer_Mercer}Let $\phi:\mathbb{B}_{N}\rightarrow\mathbb{B}_{N}$
be analytic. Then\begin{equation}
\mu_{\phi}^{\alpha}\left(S\left(\xi,h\right)\right)=O_{h\rightarrow0}\left(h^{\alpha+2}\right)\label{eq|MAC_MERCER_BERGMA_eq}\end{equation}
for every $\xi\in\mathbb{S}_{N}$.
\end{stprop}
In fact, this result is stated for general strongly pseudo-convex
domains instead of $\mathbb{B}_{N}$ (\cite[Proposition 4]{MACCLUER-MERCER}.)

A brief comparison of Condition \prettyref{eq|MAC_MERCER_BERGMA_eq}
and Condition \prettyref{eq|K_Berg_Orlicz}, written for $\psi_{1}=\psi_{2}=\psi$,
makes it clear that if we can find some $\psi$, among those satisfying
the uniform $\nabla_{0}$-Condition, which satisfies the following
condition $\mathcal{P}$:
\begin{description}
\item [{$\mathcal{P}$}] for every $K>0$, there exist $A>0$ and $h_{0}>0$
such that\begin{equation}
Kh^{\alpha+2}\leq\frac{1}{\psi\left(A\psi^{-1}\left(1/h^{N+1+\alpha}\right)\right)},\label{eq|alw_cont_Berg_1-1}\end{equation}
for any $0<h\leq h_{0}$,
\end{description}
then every composition operator will be bounded on the Bergman-Orlicz
space $A_{\alpha}^{\psi}\left(\mathbb{B}_{N}\right)$. The next proposition
characterizes those Orlicz functions which satisfy this condition
$\mathcal{P}$:
\begin{stprop}
\label{prop|alw_cont_Berg-1}Let $\psi$ be an Orlicz function. $\psi$
satisfies Condition $\mathcal{P}$ if and only if, for every $K>0$
(or equivalently for \textup{one} $K>0$), there exists $C>0$ such
that, for every $x>0$ large enough, we have\[
{\displaystyle \psi\left(x\right)^{\frac{N+1+\alpha}{\alpha+2}}}\leq K\psi\left(Cx\right).\]
In particular, Condition $\mathcal{P}$ is trivial if $N=1$ and coincides
with the $\Delta^{2}$-Condition whenever $N>1$.\end{stprop}
\begin{proof}
The first part comes from a straightforward rewritening of inequality
\prettyref{eq|alw_cont_Berg_1-1}. The second part is a direct application
of \prettyref{prop|prop_equiv_delta_up_2}, using convexity of $\psi$.
\end{proof}
\smallskip{}
When $N=1$, \cite[Theorem 3.1]{QUEF-LI-LEFEVRE-BERGMAN-ORLICZ} permits
to remove the necessary uniform $\nabla_{0}$-Condition in the first
point of \prettyref{thm|caract_cont_comp_comp_op_Berg}. When $N>1$,
this trick fails as it is not difficult to check that if it could
be extended to the several complex variables setting, then it would
imply that every composition operator is bounded on $A^{p}\left(\mathbb{B}_{N}\right)$.
Yet, we know (\prettyref{prop|nabla_0_implies_nabla_2}) that every
Orlicz function satisfying the $\Delta^{2}$-Condition satisfies the
uniform $\nabla_{0}$-Condition too.

Therefore, \prettyref{thm|caract_cont_comp_comp_op_Berg}, \prettyref{prop|MacCluer_Mercer}
and \prettyref{prop|alw_cont_Berg-1} immediately yields the following
result:
\begin{stthm}
\label{thm|B_O_compo_op_alw_bounded_delta_up_2_et_1_var}Let $\psi$
be an Orlicz function.
\begin{enumerate}
\item Every composition operator is bounded from $A_{\alpha}^{\psi}\left(\mathbb{D}\right)$
into itself;
\item When $N>1$, if $\psi$ satisfies the $\Delta^{2}$-Condition, then
every composition operator is bounded from $A_{\alpha}^{\psi}\left(\mathbb{B}_{N}\right)$
into itself.
\end{enumerate}
\end{stthm}

\end{document}